\documentclass[a4paper,11pt]{article}

\usepackage{amsthm,amsmath,amssymb}
\usepackage[latin1]{inputenc}
\usepackage[english]{babel}
\usepackage{mathscinet}
\usepackage{graphics}
\usepackage{graphicx}
\usepackage[all,cmtip]{xy}
\usepackage{multirow}
\usepackage{hyperref}

\theoremstyle{plain}
\newtheorem{theo}{Theorem}[section]
\newtheorem{prop}[theo]{Proposition}

\newtheorem{lemma}[theo]{Lemma}

\theoremstyle{definition}
\newtheorem{defi}[theo]{Definition}
\newtheorem{ex}[theo]{Example}

\theoremstyle{remark}
\newtheorem{remark}[theo]{Remark}
\newtheorem{notation}[theo]{Notation}

\DeclareMathOperator{\lcm}{lcm}
\DeclareMathOperator{\Sing}{Sing}

\DeclareMathOperator{\Gal}{Gal}
\DeclareMathOperator{\Mat}{Mat}
\DeclareMathOperator{\Spec}{Spec}
\DeclareMathOperator{\gr}{Gr}

\def\Z{\mathbb{Z}}
\def\N{\mathbb{N}}
\def\C{\mathbb{C}}
\def\Q{\mathbf{Q}}
\def\R{\mathbb{R}}

\def\P{\mathbb{P}}
\def\w{\omega}
\def\l{\ell}
\def\E{\mathcal{E}}

\newcommand\bd{\mathbf{d}}

\newcommand\bxi{\boldsymbol{\xi}}

\title{\bf Semistable Reduction of a Normal Crossing $\mathbb{Q}$-Divisor}
\author{Jorge Mart\'{\i}n-Morales\footnote{Partially supported by the Spanish Ministry of Education MTM2010-21740-C02-02, E15 Grupo Consolidado Geometría from the Gobierno de Aragón, FQM-333 from Junta de Andalucía, and PRI-AIBDE-2011-0986 Acción Integrada hispano-alemana.}}
\date{\footnotesize Centro Universitario de la Defensa - IUMA.\\ Academia General Militar, Ctra.~de Huesca s/n.\\ 50090, Zaragoza, Spain.\\ jorge@unizar.es}

\begin{document}

\maketitle

\vspace{-0.25cm}

\begin{abstract}
In a previous work we have introduced the notion of embedded $\Q$-resolution, which essentially consists in allowing the final ambient space to contain abelian quotient singularities, and A'Campo's formula was calculated in this setting. Here we study the semistable reduction associated with an embedded $\mathbf{Q}$-resolution so as to compute the mixed Hodge structure on the cohomology of the Milnor fiber in the isolated case using a generalization of Steenbrink's spectral sequence. Examples of Yomdin-L\^{e} surface singularities are presented as an application.

\vspace{0.25cm}


\vspace{0.2cm}

\end{abstract}


\section*{Introduction}

One of the main invariants of a given hypersurface singularity is the mixed Hodge
structure (MHS) on the cohomology of the Milnor fiber. In the isolated case, 
Steenbrink \cite{Steenbrink77} gave a method for computing this Hodge structure
using a spectral sequence that is constructed from the divisors associated with
the semistable reduction of an embedded
resolution, cf.~\cite{Varchenko80, Varchenko81}.

However, in practice the combinatorics of the exceptional divisor of the resolution 
is often so complicated that the study of the spectral sequence becomes very hard, 
see e.g.~\cite{Artal94} where an embedded resolution and its associated semistable 
reduction for superisolated surface singularities is computed using blow-ups at 
points and rational curves.

After the semistable reduction process the new ambient space contains normal 
singularities which are obtained as the quotient of a ball in $\C^n$ by the linear 
action of a finite group. Spaces admitting only such singularities are called {\em 
$V$-manifolds}. They were introduced in~\cite{Satake56} and have the same 
homological properties over $\mathbb{Q}$ as manifolds, e.g.~they admit a Poincaré 
duality if they are compact and carry a pure Hodge structure if they are compact 
and Kähler \cite{Baily56}. Moreover, a natural notion of normal crossing 
divisor can be defined on $V$-manifolds \cite{Steenbrink77}.

Motivated by this fact and in order to try to simplify the combinatorics of the 
exceptional divisor mentioned above, we introduced the notion of \emph{embedded $\Q$-resolution}~\cite{Martin11PhD}. The idea is as follows. Classically an embedded resolution of $\{f=0\} 
\subset \C^{n+1}$ is a proper map $\pi: X \to (\C^{n+1},0)$ from a smooth variety 
$X$ satisfying, among other conditions, that $\pi^{*}(\{f=0\})$ is a normal 
crossing divisor. To weaken the condition on the preimage of the singularity we 
allow the new ambient space $X$ to contain abelian quotient singularities and the 
divisor $\pi^{*}(\{f=0\})$ to have normal crossings on $X$.

Hence the motivation for using embedded $\Q$-resolutions rather than 
standard ones is twofold. On the one hand, they are a natural generalization of the 
usual embedded resolutions, for which the invariants above are expected to be 
calculated effectively. On the other hand, the combinatorial and computational 
complexity of embedded $\Q$-resolutions is much simpler, but they keep as much 
information as needed for the understanding of the topology of the singularity.

For instance, the behavior of the Lefschetz numbers and the zeta function of the
monodromy in this setting was treated in~\cite{Martin11}
providing the corresponding A'Campo's formula~\cite{ACampo75}. Also, for plane
curves, the local $\delta$-invariant and explicit formulas for the
self-intersections numbers of the exceptional divisors were calculated in
\cite{CMO12} and \cite{AMO11b} respectively.

In this paper we continue our study about embedded $\Q$-resolutions.
In particular, the semistable reduction of a normal crossing
$\mathbb{Q}$-divisor on an abelian quotient singularity is investigated. 
The main idea behind this construction, as mentioned above, is that in the classical case after the semistable reduction the ambient space
already contains quotient singularities. Our main result,
Theorem~\ref{SE_reduction}, says that the same is true for embedded
$\Q$-resolutions and hence Steenbrink's arguments can be
adapted to construct construct a spectral sequence converging to the cohomology
of the Milnor fiber thus providing a MHS on
$H^q(F,\C)$, see Theorem~\ref{main_th_steenbrink}.
Since the embedded $\Q$-resolution can be chosen so that ``almost every''
exceptional divisor contributes to the complex monodromy, our spectral sequence is 
finer in the sense that fewer divisors appear in the semistable reduction and thus 
the combinatorics of the spectral sequence will be simpler.

As a by-product we show that the Jordan blocks of
maximal size in the monodromy are easily calculated just by looking at the
dual complex associated with the semistable reduction of a $\Q$-resolution,
see Proposition~\ref{maximal_Jordan_blocks}.

Note that the tools developed in \cite{GLM97} for the monodromy zeta function
can not be generalized for computing more involved invariants as the MHS of the Milnor fiber.

This work in combination with~\cite{Martin12} can be considered as the first
steps in the computation of MHS, together with the monodromy
action, of the so-called Yomdin-L\^{e} surface singularities (YLS)~\cite{Yomdin74}.
Note that, following the
ideas of~\cite{Artal94}, the generalized Steenbrink's spectral sequence presented here can be used
to find two YLS having the same characteristic polynomials, the same abstract
topologies, but different embedded topologies (it is enough to take a Zariski
pairs in the tangent cones).

This paper is organized as follows. In~\S\ref{sec:preliminaries}, some well-known
preliminaries about weighted blow-ups and embedded $\Q$-resolutions are presented.
The main result, namely Theorem~\ref{SE_reduction} is proven
in~\S\ref{sec:sem_reduction}.
After recalling the monodromy filtration in~\S\ref{sec:mon_filtration},
the generalized Steenbrink's spectral sequence converging to $H^q(F,\C)$
is described in~\S\ref{sec:spectral_sequence}. Finally, as an application,
the use of all the results of this work are illustrated in~\S\ref{sec:examples}
with several examples including a plane curve and a YLS.
In particular, we provide infinite pairs of irreducible YLS having the same
complex monodromy with different topological type.

\vspace{0.25cm}

\noindent \textbf{Acknowledgments.} This is part of my PhD thesis. I am deeply grateful to my advisors Enrique Artal and José Ignacio Cogolludo for supporting me continuously with their fruitful conversations and ideas.

\section{Preliminaries}\label{sec:preliminaries}

Let us sketch some definitions and properties about $V$-manifolds, weighted projective spaces, and weighted blow-ups, see \cite{AMO11b, Martin11PhD} for a more detailed exposition.

\begin{defi}
Let $H=\{f=0\}\subset \C^{n+1}$. An {\em embedded $\Q$-resolution} of $(H,0) \subset (\C^{n+1},0)$ is a proper analytic map $\pi: X \to (\C^{n+1},0)$ such that:
\begin{enumerate}
\item $X$ is a $V$-manifold with abelian quotient singularities.
\item $\pi$ is an isomorphism over $X\setminus \pi^{-1}(\Sing(H))$.
\item \label{cond3} $\pi^{*}(H)$ is a hypersurface with $\mathbb{Q}$-normal crossings on $X$.
\end{enumerate}
\end{defi}

To deal with these resolutions, some notation needs to be introduced. Let $G := \mu_{d_0} \times \cdots \times \mu_{d_r}$ be an arbitrary finite abelian group written as a product of finite cyclic groups, that is, $\mu_{d_i}$ is the cyclic group of $d_i$-th roots of unity. Consider a matrix of weight vectors
$$
A := (a_{ij})_{i,j} = [{\bf a}_0 \, | \, \cdots \, | \, {\bf a}_n ] \in Mat ((r+1) \times (n+1), \Z)
$$
and the action
\begin{equation*}
\begin{array}{c}
( \mu_{d_0} \times \cdots \times \mu_{d_r} ) \times \C^{n+1}  \longrightarrow  \C^{n+1}, \\[0.15cm]
\big( \bxi_{\bd} , {\bf x} \big)  \mapsto  (\xi_{d_0}^{a_{00}} \cdots \xi_{d_r}^{a_{r0}}\, x_0,\, \ldots\, , \xi_{d_0}^{a_{0n}} \cdots \xi_{d_r}^{a_{rn}}\, x_n ).
\end{array}
\end{equation*}
The set of all orbits $\C^{n+1} / G$ is called ({\em cyclic}) {\em quotient space of type $({\bf d};A)$} and it is denoted by
$$
  X({\bf d}; A) := X \left( \begin{array}{c|ccc} d_0 & a_{00} & \cdots & a_{0n}\\ \vdots & \vdots & \ddots & \vdots \\ d_r & a_{r0} & \cdots & a_{rn} \end{array} \right).
$$

The orbit of an element $(x_0,\ldots,x_n)$ under this action is denoted by $[(x_0,\ldots,x_n)]$. Condition~\ref{cond3} of the previous definition means the total transform $\pi^{-1}(H) = (f\circ \pi)^{-1}(0)$ is locally given by a function of the form $x_0^{m_0} \cdots x_k^{m_k} : X({\bf d};A) \rightarrow \C$, see~\cite{Steenbrink77}. The previous numbers $m_{i}$'s have no intrinsic meaning unless $\mu_{{\bf d}}$ induces a small action on $GL(n+1,\C)$. This motivates the following.

\begin{defi}\label{def_normalized_XdA_intro}
The type $({\bf d}; A)$ is said to be {\em normalized} if the action is free on $(\C^{*})^{n+1}$ and $\mu_{\bf d}$ is identified with a small subgroup of $GL(n+1,\C)$.
\end{defi}

As a tool for finding embedded $\Q$-resolutions one uses weighted blow-ups with smooth center. Special attention is paid to the case of dimension 2 and 3 and blow-ups at points. 

\begin{ex}\label{blowup_dim2}
Assume $(d;a,b)$ is normalized and $\gcd (\w) =1$, $\w := (p,q)$. Then, the total space of the $\w$-blow-up at the origin of $X(d;a,b)$,
\begin{equation}\label{w-blow-up_intro}
\pi_{(d;a,b),\w}: \widehat{X(d;a,b)}_{\w} \longrightarrow X(d;a,b),
\end{equation}
can be written as
$$
\widehat{U}_1 \cup \widehat{U}_2 = X \left( \frac{pd}{e}; 1, \frac{-q+\beta p b}{e} \right) \cup X \left( \frac{qd}{e}; \frac{-p+\mu qa}{e}, 1 \right)
$$
and the charts are given by
\begin{equation*}
\begin{array}{c|c}
\text{First chart} & X \left( \displaystyle\frac{pd}{e}; 1, \frac{-q+\beta p b}{e} \right)  \ \longrightarrow \ \widehat{U}_1, \\[0.5cm] & \,\big[ (x^e,y) \big] \mapsto \big[ ((x^p,x^q y),[1:y]_{\w}) \big]_{(d;a,b)}. \\ \multicolumn{2}{c}{} \\
\text{Second chart} & X \left( \displaystyle\frac{qd}{e}; \frac{-p+\mu qa}{e}, 1 \right) \ \longrightarrow \ \widehat{U}_2, \\[0.5cm] & \hspace{0.15cm} \big[ (x,y^e) \big] \mapsto \big[ ((x y^p, y^q),[x:1]_{\w}) \big]_{(d;a,b)}.
\end{array}
\end{equation*}
Above, $e=\gcd(d,pb-qa)$ and $\beta a \equiv \mu b \equiv 1$ $(\text{mod $d$})$. Observe that the origins of the two charts are cyclic quotient singularities; they are located at the exceptional divisor $E$ which is isomorphic to $\P^1_{\w} \cong \P^1$.
\end{ex}

\begin{ex}\label{blowup_dim3_smooth}
Let $\pi_{\w}: \widehat{\C}^3_{\w} \to \C^3$ be the $\w$-weighted blow-up at the origin with $\w=(p,q,r)$, $\gcd(\w)=1$. The new space is covered by three open sets 
$$
\widehat{\C}^3_{\w} = U_1 \cup U_2 \cup U_3 = X(p;-1,q,r) \cup X(q;p,-1,r) \cup X(r;p,q,-1),
$$
and the charts are given by
\begin{equation}\label{charts_dim3}
\begin{array}{cc}
X(p;-1,q,r) \longrightarrow U_1: & [(x,y,z)] \mapsto ((x^p, x^q y, x^r z),[1:y:z]_{\w}), \\[0.25cm]
X(q;p,-1,r) \longrightarrow U_2: & [(x,y,z)] \mapsto ((x y^p,y^q,y^r z),[x:1:z]_{\w}), \\[0.25cm]
X(r;p,q,-1) \longrightarrow U_3: & [(x,y,z)] \mapsto ((x z^p, y z^q, z^r),[x:y:1]_{\w}).
\end{array}
\end{equation}

In general $\widehat{\C}^3_{\w}$ has three lines of (cyclic quotient) singular points located at the three axes of the exceptional divisor $\pi^{-1}_{\w}(0) \simeq \P^2_{\w}$. For instance, a generic point in $x=0$ is a cyclic point of type $\C\times X(\gcd(q,r);p,-1)$.
Note that although the quotient spaces are represented by normalized types, the exceptional divisor can still be simplified:
\begin{equation}\label{propPw}
\begin{array}{rcl}
\P^2(p,q,r) & \longrightarrow & \P^2 \displaystyle\left(\frac{p}{(p,r)\cdot
(p,q)},\frac{q}{(q,p)\cdot (q,r)},
\frac{r}{(r,p)\cdot (r,q)}\right),\\[0.5cm]
\displaystyle \,[x:y:z] & \mapsto & [x^{\gcd(q,r)}:y^{\gcd(p,r)}:z^{\gcd(p,q)}].
\end{array} 
\end{equation}

However, this simplification may be not useful when working with the whole ambient space because its charts are not compatible with $\widehat{\C}^3_{\w}$. Thus the natural covering of the exceptional divisor is
$$
  \P^2_{\w} = V_1 \cup V_2 \cup V_3 = X(p;q,r) \cup X(q;p,r) \cup X(r;p,q),
$$
and the charts are given by the restrictions of the maps in~\eqref{charts_dim3} to $x=0$, $y=0$, and $z=0$ respectively.
\end{ex}

\begin{ex}\label{blowup_dim3_singular}
Assume $(d;a,b,c)$ is normalized and $\gcd (\w) =1$, $\w := (p,q,r)$. Then, the total space of the $\w$-blow-up at the origin of $X(d;a,b,c)$,
$$
\pi = \pi_{(d;a,b,c),\w}:\, \widehat{X(d;a,b,c)}_{\w} \longrightarrow X(d;a,b,c)
$$
can be covered by three open sets as
$$
\widehat{X(d;a,b,c)}_{\w} = \frac{\widehat{\C}^3_{\w}}{\mu_d} = \frac{U_1 \cup U_2 \cup U_3}{\mu_d} = \widehat{U}_1 \cup \widehat{U}_2 \cup \widehat{U}_3,
$$
where

$$
\begin{array}{ccc}
\displaystyle \widehat{U}_1 = \frac{U_1}{\mu_d} = \frac{X(p;-1,q,r)}{\mu_d} =
X \left(\begin{array}{c|ccc} p & -1 & q & r \\ pd & a & pb-qa & pc-ra \end{array}\right), \\[0.75cm]
\displaystyle \widehat{U}_2 = \frac{U_2}{\mu_d} = \frac{X(q;p,-1,r)}{\mu_d} =
X \left(\begin{array}{c|ccc} q & p & -1 & r \\ qd & qa-pb & b & qc-rb \end{array}\right), \\[0.75cm]
\displaystyle \widehat{U}_3 = \frac{U_3}{\mu_d} = \frac{X(r;p,q,-1)}{\mu_d} =
X \left(\begin{array}{c|ccc} r & p & q & -1 \\ rd & ra-pc & rb-qc & c \end{array}\right). 
\end{array}
$$

The charts are given by the induced maps on the corresponding quotient spaces, see Equation~(\ref{charts_dim3}). The exceptional divisor $E = \pi^{-1}_{(d;a,b,c),\w}(0)$ is identified with the quotient
$$
\P^2_{\w}(d;a,b,c) := \frac{\P^2_{\w}}{\mu_d}.
$$
There are three lines of quotient singular points in $E$ and outside $E$ the map $\pi_{(d;a,b,c),\w}$ is an isomorphism.


The expression of the quotient spaces can be modified as follows. Let $\alpha$ and $\beta$ be two integers such that $\alpha d + \beta a = \gcd(d,a)$, then one has that the space $X\left(\begin{smallmatrix} p ; & -1 & q & r \\ pd ; & a & pb-qa & pc-ar \end{smallmatrix} \right)$ equals
$$
X\left(\begin{array}{c|ccc}
pd & (d,a) & -q (d,a) + \beta pb & -r (d,a) + \beta pc \\
(d,a) & 0 & b & c
\end{array}\right). 
$$
Note that in general the previous space is not represented by a normalized type. To obtain its normalized one, follow the processes described in (I.1.3) and (I.1.9) of \cite{Martin11PhD}.
\end{ex}

\section{The Semistable Reduction}\label{sec:sem_reduction}

This tool was introduced by Mumford in \cite[pp. 53-108]{Mumford73} and roughly speaking the mission of the semistable reduction is to get a reduced divisor that provides a model of the Milnor fibration. The spectral sequence converging to the cohomology of the Milnor fiber will be defined in terms of this reduced divisor, see Section \ref{sec:spectral_sequence}. Here we present a more general approach than the needed for the Milnor fibration.

\begin{notation}\label{SE_normal_notation}
Let $X$ be a complex analytic variety and let $g:X \to D_\eta^2$ be a non-constant analytic function. Assume $X$ only has abelian quotient singularities and $g^{-1}(0)$ is a $\mathbb{Q}$-normal crossing divisor, that is, $g$ is locally given by a function of the form $x_0^{m_0} \cdots x_k^{m_k}: X(\bd;A) \rightarrow \C$. Let~$e$ be the least common multiple of all possible multiplicities appearing in the divisor $g^{-1}(0)$ and consider $\sigma: D^2_{\eta^{1/e}} \to D^2_{\eta}$ the branched covering defined by $\sigma(t) = t^e$.

Denote by $(X_1, g_1, \sigma_1)$ the pull-back of $g$ and $\sigma$.
$$
\xymatrix{
X_1 \ar[r]^{g_1} \ar[d]_{\sigma_1} & D^2_{\eta^{1/e}} \ar[d]^{\sigma}\\
X \ar[r]_{g} & D^2_{\eta}
}
$$

The map $\sigma_1$ is a cyclic covering of $e$ sheets ramified over $g^{-1}(0)$. If $F$ denotes the Milnor fiber of $g:X \to \C$, then $\sigma_1^{-1}(F)$ has $e$ connected components which are projected diffeomorphically onto $F$.

We have not yet completed the construction of the semistable reduction because~$X_1$ is not normal. Indeed, given $P\in g^{-1}(0)$ there exist integers $k \geq 0$ and $m_0,\ldots,m_k \geq 1$ such that
$$
g(x_0,\ldots,x_n) = x_0^{m_0}\cdots x_k^{m_k}: B^{2n+2}/\mu_\bd \longrightarrow \C,
$$
where $B^{2n+2}$ is an open ball of $\C^{n+1}$ and the group $\mu_\bd$ acts diagonally as in $(\bd;A)$. Denote by $P_1$ the unique point in $\sigma_1^{-1}(P)$. Then, $X_1$ in a neighborhood of $P_1$ is of the form
\begin{equation}\label{expression_X1}
  \Big\{ \big( \big[(x_0,\ldots,x_n)\big], t \big) \in X(\bd;A) \times \C \quad \big| \quad t^e = x_0^{m_0} \cdots x_k^{m_k} \Big\},
\end{equation}
and hence the space $X_1$ is not necessarily normal.

Let $\nu: \widetilde{X} \to X_1$ be the normalization and denote by $\widetilde{g} := g_1\circ \nu$ and $\varrho := \sigma_1 \circ \nu$ the natural maps. The normalization process has essentially two steps when the corresponding ring is a unique factorization domain (UFD). First, separate the irreducible components, and then find the normalization of each component. In the latter case, the ring in question is a domain and the following result applies.
\end{notation}

\begin{lemma}\label{SE_normal_domain}
Let $A \subset B$ be an integral extension of commutative rings. Suppose that $B$ is an integrally closed domain such that $Q(B)|Q(A)$ is a Galois extension. Then, the normalization of the ring~$A$ is $\overline{A} = B^{\Gal(Q(B)|Q(A))}$.
\end{lemma}

\begin{proof}
Since $B$ is normal and the extension $A\subset B$ is integral, then $\overline{A} = B \cap Q(A)$. Now the statement follows from the Galois condition.
\end{proof}

\begin{ex}\label{XA_not_UFD}
The algebraic ring of functions of $X(2;1,1)$ is isomorphic to $\C[x^2, xy, y^2]$ as an algebraic variety. In this ring the polynomial $xy$ is irreducible but not prime. To compute the normalization of the quotient ring $\C[x^2,xy,y^2]/\langle xy \rangle$, one can not proceed in the same way as in a UFD. This happens because $\mu_2$ does not define an action on the factors of the polynomial $xy$.
\end{ex}

Although the ring of functions of the previous space (\ref{expression_X1}) is not a UFD, see Example~\ref{XA_not_UFD} above, to compute the normalization of $X_1$ one can proceed in the same spirit because of the special form of the polynomial $t^e - x_0^{m_0} \cdots x_k^{m_k}$, see proof of Theorem~\ref{SE_reduction}. Before that we need to introduce some notations.

\begin{defi}\label{def_mult_2}
Let $X$ be a complex analytic space having only abelian quotient singularities and consider $E$ a $\mathbb{Q}$-normal crossing divisor on $X$. Assume $P\in |E|$ is a point such that the local equation of $E$ at $P$ is given by the function
$$
x_0^{m_0} \cdots x_k^{m_k} :\, X(\bd; A) := \C^{n+1}/\mu_\bd \longrightarrow \C, \quad (0 \leq k \leq n)
$$
where $x_0, \ldots, x_n$ are local coordinates of $X$ at $P$, $\bd=(d_0,\ldots,d_r)$, and $A = (a_{ij})_{i,j} \in \Mat((r+1)\times (n+1),\Z)$.

The {\em multiplicity} of $E$ at~$P$, denoted by $m(E,P)$ or simply $m(P)$ if the divisor is clear from de context, is defined by
$$
m(E,P):= \gcd \bigg( m_0, \ldots, m_k, \frac{\sum_{j=0}^k a_{0j} m_j}{d_0}, \ldots, \frac{\sum_{j=0}^k a_{rj} m_j}{d_r} \bigg).
$$

If there exists $T\subset |E|$ such that the function $P\in T \mapsto m(E,P)$ is constant, then we use the notation $m(T) := m(E,P_0)$, where $P_0$ is an arbitrary point in $T$.
\end{defi}

\begin{remark}\label{def1_equals_def2}
Using the general fact $\lcm (\frac{m}{b_0}, \ldots, \frac{m}{b_r}) = \frac{m}{\gcd(b_0,\ldots,b_r)}$, one easily checks that this definition coincides with the one of~\cite[Def.~2.6]{Martin11} for $k=0$, cf.~\eqref{mp_calculation}, that is,
$$
m(E,P):= \frac{m}{L}, \qquad L = \lcm \left( \frac{d_0}{\gcd(d_0,a_{00})},\ldots,\frac{d_r}{\gcd(d_r,a_{r0})} \right),
$$
where $E$ is a $\mathbb{Q}$-divisor on $X$ locally given at the point $P$ by the well-defined function $x_0^m: X(\bd;A) \to \C$.
\end{remark}

In the situation of~\ref{SE_normal_notation}, the multiplicity $m(g^{*}(0),P)$ with $P\in g^{-1}(0)$ can be interpreted geometrically as follows.

\begin{lemma}\label{prime_factors}
The number of prime (or irreducible) factors of the polynomial $t^e - x_{0}^{m_0} \cdots x_k^{m_k}$ regarded as an element in $\C[x_0,\ldots,x_n]^{\mu_\bd} \otimes_{\C} \C[t]$ is $m(g^{*}(0),P)$. Hence this number also coincides with the cardinality of the fiber over $P$ of the covering $\varrho: \widetilde{X} \to X$.
\end{lemma}

\begin{proof}
Let us denote $\l = \gcd(m_0,\ldots,m_k)$ and $C_i = \sum_{j=0}^k a_{ij} m_j$ for $i=0,\ldots,r$. The polynomial $t^e - x_0^{m_0} \cdots x_k^{m_k} \in \C[x_0, \ldots, x_n, t]$ factorizes into $\l$ different components as
$$
t^e - x_0^{m_0} \cdots x_k^{m_k} = \prod_{i=0}^{\l-1} \Big( t^{\frac{e}{\l}} - \zeta^i_\l\, x_0^{\frac{m_0}{\l}} \cdots x_k^{\frac{m_k}{\l}} \Big),
$$
where $\zeta_\l$ is a primitive $\l$-th root of unity. However, this factors are not invariant under the group~$\mu_\bd$, since they are mapped to
$$
t^{\frac{e}{\l}} - \zeta^i_\l\, x_0^{\frac{m_0}{\l}} \cdots x_k^{\frac{m_k}{\l}} \quad \longmapsto \quad t^{\frac{e}{\l}} - \xi_{d_0}^{\frac{C_0}{\l}} \cdots \xi_{d_r}^{\frac{C_r}{\l}} \cdot \zeta^i_\l\, x_0^{\frac{m_0}{\l}} \cdots x_k^{\frac{m_k}{\l}},
$$
by the action of $(\xi_{d_0}, \ldots, \xi_{d_r}) \in \mu_\bd$. Recall that $\C^{n+1}/\mu_\bd = X(\bd; A)$.

Let $H_i$ be the cyclic group defined by $H_i:= \{ \xi_{d_i}^{C_i/\l} \mid \xi_{d_i} \in \mu_{d_i} \}$, for $i=0,\ldots,r$, and consider $H = H_0 \cdots H_r$. Since $t^e - x_0^{m_0} \cdots x_k^{m_k}$ defines a function over $X(\bd; A)\times \C$, then $d_i$ must divide $C_i$ and, consequently, all the previous groups are (normal) subgroups of $\mu_{\l}$. The order of~$\mu_{\l} / H$ is exactly the number of prime (or irreducible) components of the preceding polynomial regarded as an element in $\C[x_0,\ldots,x_n]^{\mu_\bd} \otimes_{\C} \C[t]$.

The order of $H_i$ is $|H_i| = \frac{d_i}{\gcd ( d_i,\, C_i/\l )} = \frac{\l}{\gcd ( \l,\, C_i/d_i )}$. Then, one has
\begin{equation}\label{mp_calculation}
\begin{split}
|H| & = | H_0 \cdots H_r | = \lcm \big( |H_0|, \ldots, |H_r| \big) \\ & = \frac{\l}{ \gcd \left( \l, \frac{C_0}{d_0}, \ldots, \frac{C_r}{d_r} \right) } = \frac{\l}{m(P)}.
\end{split}
\end{equation}
In the expression above, a general property about greatest common divisor and least common multiple already mentioned in~\ref{def1_equals_def2} was used.
\end{proof}

Assume that $g^{-1}(0) = E_0 \cup \cdots \cup E_s$ and let us denote $D_i = \varrho^{-1}(E_i)$ for $i=0,\ldots,s$ and $D = \bigcup_{i=0}^s D_i$. This commutative diagram illustrates the whole process of the semistable reduction.
\begin{equation}\label{SE_normal_diagram}
\xymatrix{
D_i \ar@{^{(}->}[r] \ar[d]_{\varrho} & \widetilde{X} \ar[r]^{\nu} \ar[d]_{\varrho} \ar@/^0.65cm/[rr]^{\widetilde{g}} & X_1 \ar[r]^{g_1} \ar[d]_{\sigma_1} & D^2_{\eta^{1/e}} \ar[d]^{\sigma}\\
E_i \ar@{^{(}->}[r] & X \ar@{=}[r] & X \ar[r]_{g} & D^2_{\eta}
}
\end{equation}

Consider the stratification of $X$ associated with the normal crossing divisor $g^{-1}(0) \subset X$. That is, given a possibly empty set $I\subseteq \{0,1,\ldots,s\}$, consider
$$
  E_{I}^\circ := \Big( \cap_{i \in I} E_i \Big) \setminus \Big( \cup_{i\notin I} E_i \Big).
$$
Also, let $X = \bigsqcup_{j\in J} Q_j$ be a finite stratification of $X$ given by its quotient singularities so that the local equation of $g$ at $P \in E_I^{\circ} \cap Q_j$ is of the form
$$
  x_1^{m_1} \cdots x_k^{m_k}:\, B/G \longrightarrow \C,
$$
where $B$ is an open ball around $P$, and $G$ is an abelian group acting diagonally as in $(\bd;A)$. The multiplicities $m_i$'s and the action $G$ are the same along each stratum $E_I^{\circ} \cap Q_j$, i.e.~they do not depend on the chosen point $P \in E_I^{\circ} \cap Q_j$. Denote $m_{I,j} := m(E_I^{\circ} \cap Q_j)$. Finally, assume that $E^\circ_I \cap Q_j$ is connected.


\begin{theo}\label{SE_reduction}
The variety $\widetilde{X}$ only has abelian quotient singularities located at $\widetilde{g}^{-1}(0) = D$ which is a reduced divisor with normal crossings on~$\widetilde{X}$.
Also, $\varrho: \widetilde{X} \to X$ is a cyclic branched covering of $e$ sheets unramified over~$X\setminus g^{-1}(0)$.
Moreover, for $\emptyset \neq I \subseteq S:= \{0,1,\ldots,s\}$ and $j \in J$, the following properties hold. \vspace{0.1cm}
\begin{enumerate}\setlength{\itemsep}{5pt}
\item The restriction $\varrho\,|: \varrho^{-1}(\overline{E^\circ_I\cap Q_j}) \rightarrow \overline{E_I^\circ \cap Q_j}$ is a cyclic branched covering of $m_{I,j}$ sheets unramified over $E^\circ_I \cap Q_j$.
\item The variety $\varrho^{-1}(\overline{E^\circ_I\cap Q_j})$ is a $V$-manifold with abelian quotient singularities with $\gcd( \{ m(P) \mid P\in \overline{E^\circ_I\cap Q_j} \} )$ connected components.
\item Let $\varphi: \widetilde{X} \to \widetilde{X}$ be the canonical generator of the monodromy of the covering~$\varrho$. Then, its restriction to $\varrho^{-1}(\overline{E^\circ_I\cap Q_j})$ is a generator of the monodromy of $\varrho\,|: \varrho^{-1}(\overline{E^\circ_I\cap Q_j}) \rightarrow \overline{E_I^\circ \cap Q_j}$.
\item The Euler characteristic of each connected component of $D_i$ is
$$
  \qquad \displaystyle \sum_{\begin{subarray}{c} \{ i \} \subset I\subset \{ 0,1,\ldots,s \} \\ j \, \in \, J \end{subarray}} m_{I,j} \cdot \chi(E_I^\circ \cap Q_j) \bigg/ \gcd( \{ m(P) \mid P\in E_i \} ).
$$
\end{enumerate}
\end{theo}

\begin{proof}
First note that the morphism $\varrho: \widetilde{X} \to X$ is a cyclic branched covering unramified over $X\setminus g^{-1}(0)$, since so is $\sigma_1: X_1 \to X$ and the normalization $\nu: \widetilde{X} \to X_1$ does not change the normal points.

Let $P \in g^{-1}(0)$ and choose coordinates $x_0, \ldots, x_n$ as in~\ref{SE_normal_notation} so that $X_1 \subset X(\bd;A) \times \C$ is locally given by the polynomial $t^{e} - x_0^{m_0} \cdots x_k^{m_k}$. Let us denote for $i = 0, \ldots, k$,
$$
m(P) = m(g^{*}(0),P), \qquad e' = e/m(P), \qquad m'_i = m_i/m(P).
$$

Consider the ring
$$
  A = \frac{\C[x_0,\ldots,x_n, t]}{\langle t^{e} - x_0^{m_0} \cdots x_k^{m_k} \rangle}.
$$

The action given by $X(\bd;A)$ is extended to $A$ so that the variable $t$ is invariant. Then, by~Lemma~\ref{prime_factors}, the normalization $\overline{A^{\mu_\bd}}$ of the ring $A^{\mu_\bd}$ is isomorphic to the direct sum of $m(P)$ isomorphic copies of the normalization~of
$$
  \frac{\C[x_0,\ldots,x_n]^{{\mu_\bd}} \otimes_{\C} \C[t]}{ \big\langle t^{e'} - x_0^{m'_0} \cdots x_k^{m'_k} \big\rangle } =
  \Bigg( \frac{\C[x_0,\ldots,x_n,t]}{ \big\langle t^{e'} - x_0^{m'_0} \cdots x_k^{m'_k} \big\rangle } \Bigg)^{{\mu_\bd}}.
$$
Therefore to compute it we only need to consider the case $m(P) = 1$, for which the ring $A^{\mu_\bd}$ is an integral domain. Now we plan to apply Lemma~\ref{SE_normal_domain} to a ring extension $A^{\mu_\bd} \subset B$, where $B$ is a polynomial algebra.


Let $c_i = e/m_i$ for $i=0,\ldots,k$. Denote $B = \C[y_0,\ldots,y_n]$ and consider~$A^{\mu_\bd}$ as subring of $B$ by putting
$$
\begin{cases}
x_i = y_i^{c_i} & \text{if} \quad 0 \leq i \leq k, \\
x_i = y_i & \text{for} \quad i > k, \\
t = y_0 \cdots y_k
\end{cases}
$$

Note that $A$ can not be embedded in $B$ because it is not even a domain. Since $\mu_\bd$ acts diagonally on~$\C^{n+2}$, there exists $N \gg 0$ such that
$$
  y_0^{c_0 N}, \ldots, y_k^{c_k N},\ y_{k+1}^N, \ldots, y_n^N \ \in \ A^{\mu_\bd}.
$$

This implies that the extension $A^{\mu_\bd} \subset B$ is integral. Also, $B$ is a normal domain. It remains to prove that $Q(B)|Q(A^{\mu_\bd})$ is a Galois field extension. One has
$$
  \C(y_0^{c_0 N}, \ldots, y_k^{c_k N}, y_{k+1}^N, \ldots, y_n^N) \subset Q(A^{\mu_\bd}) \subset Q(B) = \C(y_0,\ldots,y_n).
$$
Note that the largest extension is clearly Galois. Its Galois group is abelian and it is isomorphic to
$$
\mu_{c_0 N} \times \cdots \times \mu_{c_k N} \times \mu_{N} \times \stackrel{n-k}{\ldots} \times \mu_{N}.
$$
Thus $\overline{A^{\mu_\bd}} = B^{\Gal(Q(B)|Q(A^{\mu_\bd}))}$.

This shows that $\Spec (\overline{A^{\mu_\bd}})$ and hence $\widetilde{X}$ are $V$-manifolds. Locally $D$ is the quotient under the group $\Gal(Q(B)|Q(A^{\mu_\bd}))$ of the reduced divisor $y_0 \cdots y_k = 0$. The rest of the statement follows from the fact that the branched coverings involved are cyclic. For the last part, use the classical Riemann-Hurwitz formula.
\end{proof}

\begin{remark}
Assume $\C[x_0,\ldots,x_n]^{\mu_\bd} = \C[\{ x_0^{\alpha_0} \cdots x_n^{\alpha_n} \}_{\alpha \in \Lambda}]$.
Then $A^{\mu_\bd}$ is identified with the subring
$$
\C[ \{ y_0^{\alpha_0 c_0} \cdots y_k^{\alpha_k c_k} \cdot y_{k+1}^{\alpha_{k+1}} \cdots y_n^{\alpha_n} \}_{\alpha \in \Lambda},\, y_0 \cdots y_k ] \subset \C[y_0,\ldots,y_n].
$$
Hence the Galois extension
$$
\Gal(Q(B)|Q(A^{\mu_\bd})) \subset \mu_{c_0 N} \times \cdots \times \mu_{c_k N} \times \mu_{N} \times \stackrel{(n-k)}{\ldots} \times \mu_{N}
$$
is given by the elements $(\xi_0, \ldots, \xi_k, \eta_{k+1}, \ldots, \eta_n)$ such that
$$
  \forall \alpha \in \Lambda, \quad \left\{ \begin{array}{l} \xi_0^{\alpha_0 c_0} \cdots \xi_k^{\alpha_k c_k} \cdot \eta_{k+1}^{\alpha_{k+1}} \cdots \eta_n^{\alpha_n} = 1,\\ \xi_0 \cdots \xi_k = 1. \end{array}\right.
$$

In general, this group is not a small subgroup of $GL(n+1,\C)$, that is, there may exist elements of the group having $1$ as an eigenvalue of multiplicity precisely $n$.
\end{remark}

\begin{remark}
Note that $\varrho\,|:\varrho^{-1} ( \overline{E^\circ_i \cap Q_j} ) \rightarrow \overline{E^\circ_i \cap Q_j}$ is an isomorphism when $I = \{ i \}$ and the multiplicity of $E_{i}$ (at the smooth points) is equal to one.
\end{remark}

In what follows this construction is applied to $g = f \circ \pi$, where the
map~$f:(M,0) \to (\C,0)$ is the germ of a non-constant analytic function and
$\pi: X \to (M,0)$ is an embedded $\Q$-resolution of $\{ f=0 \} \subset (M,0)$
with $M = X(\bd;A)$, cf.~Section~\ref{sec:spectral_sequence}.
Let us see an example.


\begin{ex}
Consider the plane curve defined by $f = x^p + y^q$ in $\C^2$. Recall that after the $(q_1,p_1)$-weighted blow-up at the origin, one obtains an embedded $\Q$-resolution with only one exceptional divisor $\E$ of multiplicity $\lcm(p,q)$, where $p=p_1 \gcd(p,q)$ and $q=q_1 \gcd(p,q)$, see e.g.~\cite[Ex.~3.3]{Martin11}.

Following Theorem~\ref{SE_reduction}, $D = \varrho^{-1}(\E)$ is irreducible and the restriction $\varrho: D \to \E$ is a branched covering of $\lcm(p,q)$ sheets. Also, the singular point of type $(q_1; -1, p_1)$ (resp.~$(p_1; q_1, -1)$) is converted into $p$ (resp.~$q$) smooth points in the semistable reduction. Finally, $\varrho\,|: \varrho^{-1}({\bf \widehat{C}}) \to {\bf \widehat{C}}$ is an isomorphism. This implies that the Euler characteristic of $D$ is
$$
\chi(D) = p + q + \gcd(p,q) - pq = \gcd(p,q) + 1 - \mu.
$$

\begin{figure}[ht]
\centering
\includegraphics{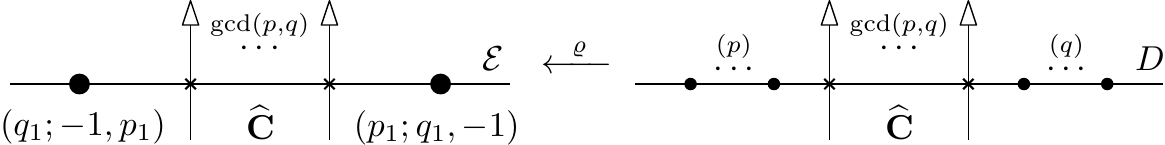}
\caption{Semistable reduction of $x^p + y^q$.}
\label{ex_210}
\end{figure}


The $p$ points in $D$ which are lift over the point of type~$(q_1;-1,p_1)$ are smooth. Of course, the same happens for the point of type $(p_1;q_1,-1)$. Also, the intersection of the strict transform with $D$ gives rise to $\gcd(p,q)$ smooth points. As we shall see the smoothness is not relevant for providing a MHS on the cohomology of the Milnor fiber.
\end{ex}

\section{Monodromy Filtration}\label{sec:mon_filtration}

This exposition is extracted from~\cite{Artal94}, which is in turn based on the book~\cite{AVG88}.

Let $H$ be a $\C$-vector space of finite dimension. Consider a nilpotent endomorphism $N: H \to H$, i.e.~there exists $k \in \N$ such that $N^k = 0$. Its Jordan canonical form is determined by the sequence of integers formed by the size of the Jordan blocks.

There is an alternative way to encode the Jordan form giving instead an increasing filtration on $H$. Let us fix $k \in \Z$; it will be called the {\em central index} of the filtration. Consider a basis of $H$ such that the matrix of $N$ in this basis is the Jordan matrix.

Each Jordan block determines a subfamily $\{ v_1,\ldots,v_r \}$ of the basis such that $N(v_1) = 0$ and $N(v_i) = v_{i-1}$ for $i=2,\ldots,r$. Let us denote by~$l(v_i)$ the unique integer 	determined by the following two conditions:
\begin{enumerate}
\item $l(v_i) = l(v_{i-1}) + 2$, $\forall i=2,\ldots,r$.
\item $\{ l(v_1), \ldots, l(v_r) \}$ is symmetric with respect to $k$.
\end{enumerate}
In fact, this integer is $l(v_i) = k-r+2i-1$, $\forall i=1,\ldots,r$, as one can check directly.

Applying this construction to all the Jordan blocks, one defines $W_l$ as the vector subspace of $H$ generated by $\{ v \mid \text{$v$ in the basis},\ l(v) \leq l \}$. This gives rise to an increasing filtration $\{ W_l \}_{l \in \Z}$ on $H$. Its graded part is denoted by $\gr_l^W (H) := W_l / W_{l-1}$ for $l \in \Z$.

Also, denote by $J_l(N)$ the number of Jordan blocks in $N$ of size $l$. Then, it is satisfied that
$$
J_l(N) = \dim ( \gr_{k-l+1}^W (H) ) - \dim ( \gr_{k-l-1}^W (H) ).
$$

\begin{prop}[\cite{Schmid73}]
There exists a unique increasing filtration $\{ W_l \}_{l \in \Z}$ such that:
\begin{enumerate}
\item $N(W_l) \subset W_{l-2}$.
\item $N^{l}: \gr_{k+l}^W (H) \to \gr_{k-l}^W (H)$ is an isomorphism.
\end{enumerate}
\end{prop}

This filtration is called the {\em weight filtration} of $N$ with central index~$k$.
One checks that the filtration $\{ W_l \}_{l\in\Z}$ defined above satisfies these two properties. In particular, the description of $\{ W_l \}_{l\in\Z}$ does not depend on the chosen basis.

\vspace{0.25cm}

Using this construction, the Jordan form of an arbitrary automorphism $M: H \to H$ can be described too. Let $M = M_u M_s$ be the decomposition of $M$ into its unipotent and semisimple components. It is known that $M_u M_s = M_s M_u$ and that the decomposition is unique, see~\cite{Serre66}. Recall that the semisimple part contains the information about the eigenvalues and the unipotent one, the information about the size of the Jordan blocks.
Note that the endomorphism $N:=\log (M_u)$ is nilpotent and the number of Jordan blocks of size $l$ is $J_l (N) = J_l(M_u) = J_l(M)$.

For a given $k \in \Z$, consider the weight filtration associated with $N$ with central index $k$. Due to the properties of the decomposition, the subspaces $W_l$ are invariant by the action of $M_s$, and thus by the action of $M$.
The endomorphism induced by $M_u$ on each graded part $\gr_l^W (H)$ is semisimple and, since $M_u$ is unipotent, it is indeed the identity. Hence the actions of $M$ and $M_s$ on $\gr_l^W (H)$ coincide.

The conclusion is that the Jordan form of $M$ is determined by the filtration $\{ W_l \}_{l \in \Z}$ and the action of $M$ over $\gr_l^W (H)$ for $l\in \Z$.

\vspace{0.25cm}

Let $(V,0) \subset (\C^{n+1},0)$ be a germ of an isolated hypersurface singularity at the origin. Denote by $\varphi: H^n(F,\C) \to H^n(F,\C)$ its complex monodromy.

Consider the decomposition of $H^n(F,\C)$ as a direct sum of two subspaces invariant under $\varphi$, $H^{\neq 1}$ and $H^{1}$, such that $Id-\varphi$ is invertible over $H^{\neq 1}$ and nilpotent over $H^{1}$.

Let $W^{\neq 1}$ be the weight filtration of $\varphi|_{H^{\neq 1}}$ with central index $n$. Analogously, denote by $W^{1}$ the weight filtration of $\varphi|_{H^{1}}$ with central index $n+1$. These filtrations satisfy $W^{\neq 1}_{-1} = W^{1}_1 = 0$, $W^1_{2n} = H^1$, and $W^{\neq 1}_{2n} = H^{\neq 1}$.

\begin{defi}
The {\em monodromy filtration} of the cohomology of the Milnor fiber is $W:= W^{1} \oplus W^{\neq 1}$.
\end{defi}

Note that the Jordan form of the complex monodromy is completely determined by the action of $\varphi$ over the graded parts of the monodromy filtration $W$.
Let us fix the notation for the characteristic polynomials of $\varphi$ acting on the following vector spaces: 
$$
\begin{array}{|c|c|}
\hline
\text{{\bf Vector space}} & \text{{\bf Characteristic polynomial}} \\
\hline
H := H^n(F,\C) & \Delta(t) \\[0.25cm]
\gr_{n-l}^{W^{\neq 1}} ( H ) & \Delta_{l}^{\neq 1}(t) \\[0.25cm]
\gr_{n-l+1}^{W^1} (H) & \Delta_{l}^{1}(t) \\[0.25cm]
\gr_{n-l}^{W^{\neq 1}} (H) \oplus \gr_{n-l+1}^{W^1} (H) & \Delta_l(t) \\
\hline
\end{array}
$$


Observe that the Jordan blocks of size $l$ are given by the polynomial $\frac{\Delta_{l-1}(t)}{\Delta_{l+1}(t)}$. More precisely, the multiplicity of $\zeta \in \C$ as root is this polynomial equals the number of Jordan blocks of size $l$ for the eigenvalue~$\zeta$.

\section{Steenbrink's Spectral Sequence}\label{sec:spectral_sequence}

The Jordan form of the complex monodromy is closely related to the theory of MHS, first introduced in~\cite{Deligne71a, Deligne71b, Deligne74}. By different methods, Steenbrink and Var{\v{c}}enko proved that the cohomology of the Milnor fiber admits a MHS compatible with the monodromy, see 
\cite{Steenbrink77} and \cite{Varchenko80, Varchenko81}.

\begin{defi}
A {\em Hodge structure} of weight $n$ is a pair $(H_{\Z},F)$ consisting of a finitely generated abelian group $H_{\Z}$ and a decreasing filtration $F = \{ F^p \}_{p \in \Z}$ on $H_{\C} := H_{\Z} \otimes_{\Z} \C$ satisfying $H_\C = F^p \oplus \overline{F^{n-p+1}}$ for all $p \in \Z$. One calls $F$ the {\em Hodge filtration}.
\end{defi}


An equivalent definition is obtained replacing the Hodge filtration by a decomposition of $H_\C$ into a direct sum of complex subspaces $H^{p,q}$, where $p+q=n$, with the property that $\overline{H^{p,q}} = H^{q,p}$. The relation between these two descriptions is given by
$$
H_{\C} = \bigoplus_{p+q = n} H^{p,q}, \qquad F^p = \bigoplus_{i\geq p} H^{i,n-i}, \qquad H^{p,q} = F^p \cap \overline{F^q}.
$$

The typical example of a pure Hodge structure of weight $n$ is the cohomology $H^n(X,\Z)$ where $X$ is a compact Kähler manifold. In the sequel, we will use the fact that, for compact Kähler $V$-manifold, $H^n(X,\Z)$ can also be endowed with a pure Hodge structure of weight~$n$. Deligne proved that the same is true for smooth compact algebraic varieties, see~\cite{Deligne71b}.


Above, one may replace $\Z$ by any ring $A$ contained in $\R$ such that $A \otimes_{\Z} \mathbb{Q}$ is a field and obtain $A$-Hodge structures. In particular, one uses $A = \mathbb{Q}$ or $\R$. In this way the primitive cohomology groups of a compact Kähler manifold are $\R$-Hodge structures.

\begin{defi}
A {\em mixed Hodge structure} is a triple $(H_\Z,W,F)$ where $H_\Z$ is a finitely generated abelian group, $W=\{W_n\}_{n\in \Z}$ is an increasing filtration on $H_{\mathbb{Q}}:= H_{\Z} \otimes_{\Z} \mathbb{Q}$, and $F=\{F^p\}_{p\in\Z}$ is a decreasing filtration on $H_{\C} := H_{\Z} \otimes_{\Z} \C$, such that $F$ induces a $\mathbb{Q}$-Hodge structure of weight $n$ on each graded part $\gr_n^W (H_{\mathbb{Q}})$, $\forall n \in \Z$. One calls $F$ the {\em Hodge filtration} and $W$ the {\em weight filtration}.
\end{defi}


Let us denote again by the same letter the filtration induced by $W$ on the complexification $H_{\C}$, i.e.~$W_n(H_{\C}) = W_n \otimes \C$. Then, the filtration induced by $F$ on $\gr_n^{W} (H_\C)$ is defined by
$$
F^p \left( \gr_n^W (H_\C) \right) = \frac{F^p \cap (W_n \otimes \C) + W_{n-1} \otimes \C}{W_{n-1} \otimes \C}.
$$
Thus the condition above on the weight and Hodge filtrations can be stated as, $\forall n,p \in \Z$,
$F^p \left( \gr_n^W (H_\C) \right) \oplus \overline{F^{n-p+1} \left( \gr_n^W (H_\C) \right)} = \gr^W_n ( H_\C )$.

\begin{ex}
Let $D$ be a divisor with normal crossings whose irreducible components are smooth and Kähler. Then, $H^{*}(D,\Z)$ admits a functorial MHS, see~\cite{GS75}. This results is extended to $V$-manifolds with $\mathbb{Q}$-normal crossings whose irreducible components are Kähler.
Also, in~\cite{Deligne71b}, it is proven that if $X$ is the complement in a compact Kähler manifold of a normal crossing divisor, then $H^{*}(X,\Z)$ has a functorial MHS which does not depend on the ambient variety.
\end{ex}

Let $M = X(\bd;A) = \C^{n+1}/\mu_{\bd}$ be an abelian quotient space and consider
a non-constant analytic function germ $f:(M,0) \to (\C,0)$. Let us fix an
embedded $\Q$-resolution $\pi: X \to (M,0)$ of the hypersurface
$\{ f=0 \}$.
Hence the construction of \S\ref{sec:sem_reduction} is applied to
$g = f \circ \pi : X \to \C$.

The following result can be proven as in~\cite{Steenbrink77} repeating exactly the same arguments. The main reason is that, starting with an embedded $\Q$-resolution, the total space produced after the semistable reduction is again a $V$-manifold with abelian quotient singularities, see Theorem~\ref{SE_reduction}.

\begin{theo}\label{main_th_steenbrink}
There exists a spectral sequence $\{ E_{n}^{p,q} \}$ constructed from the embedded $\Q$-resolution $\pi$ that verifies:
\begin{enumerate}\setlength{\itemsep}{5pt}
\item It converges to the cohomology of the Milnor fiber and degenerates at~$E_2$.
\item The spaces $E_{1}^{p,q}$ has a pure Hodge structure of weight $p$ respected by the differentials. In particular, $E_{2}^{p,q} = E_{\infty}^{p,q}$ also has a pure Hodge structure of weight $p$.
\item There exists a Hodge filtration on the cohomology of the Milnor fiber which induces a Hodge filtration on $E_{\infty}^{p,q}$. One constructs a weight filtration using the filtration with respect to the first index:
$$
  \qquad \gr_l^W (H^k (F,\C)) \cong E_{\infty}^{l,k-l} \cong E_{2}^{l,k-l}.
$$
\end{enumerate}

Therefore, these two filtrations provide a MHS on the cohomology of the Milnor fibration. This structure is an invariant of the singularity which only depends on the resolution $\pi$.
\end{theo}

In~\cite{Varchenko81}, there is another construction of the MHS on the cohomology of the Milnor fiber, using asymptotic integration. The weight filtration of both MHS coincide. Var{\v{c}}enko's definition does not depend on the resolution. Although both Hodge filtrations do not coincide, they induce the same pure Hodge structure on the graded part of the weight filtration.

\begin{theo}
The complexification of the weight filtration of the MHS of the cohomology of the Milnor fiber is exactly the monodromy filtration.

Moreover, the complex monodromy $\varphi$ acts over the first term $E_1$ of the spectral sequence and commutes with the differentials. The action induced on the complexification of $E_2 = E_\infty$ coincides with the action induced on the graded parts of the monodromy filtration.
\end{theo}

We finish this section with the explicit description of Steenbrink's spectral sequence. As we shall see, it is constructed from the divisors associated with the semistable reduction of $g := f \circ \pi : X \to \C$.

Consider the divisor $D$ associated with the semistable reduction of the embedded $\Q$-resolution $\pi$. Let us decompose $D = D_0 \cup D_1 \cup \cdots \cup D_s$ so that $D_0$ corresponds to the strict transform of the singularity and the divisor $D_{+} := D_1 \cup \cdots \cup D_s$ corresponds to the exceptional components.
Let us introduce some notation.
\begin{itemize}
\item Let $I = (i_0, \ldots, i_k)$ with $0 \leq i_0 < \cdots < i_k \leq s$. \vspace{0.25cm} 
\begin{align*}
D_I = D_{i_0,\ldots,i_k} & := D_{i_0} \cap \cdots \cap D_{i_k}, \\[0.2cm]
\check{D}_I = \check{D}_{i_0,\ldots,i_k} & := D_I \setminus \bigcup_{j \neq i_0,\ldots, i_k} ( D_j \cap D_I ).
\end{align*}
The first one is a projective $V$-manifold of dimension
$n-k$. The second one is a smooth complex variety of the same dimension.

\item Let $0 \leq i_0 < \cdots < i_k \leq s$, \ $i_j < i'_j < i_{j+1}$ with $-1 \leq j \leq k$. Denote by
$$
\hspace{1.25cm} \kappa_{i_0, \ldots, i_j, i_{j+1}, \ldots, i_k}^{i'_j}\, :\,
D_{i_0,\ldots,i_j,i'_j,i_{j+1},\ldots,i_k} \lhook\joinrel\longrightarrow
D_{i_0,\ldots,i_j,i_{j+1},\ldots,i_k},
$$
the natural inclusion.

\item Let $D^{[k]}:= \displaystyle \bigsqcup_{0 \leq i_0 < \cdots < i_k \leq s} D_{i_0,\ldots,i_k}$, \qquad
$D_{+}^{[k]}:= \displaystyle \bigsqcup_{1 \leq i_0 < \cdots < i_k \leq s} D_{i_0,\ldots,i_k}$.
\end{itemize}

\begin{defi}
Let $k \in \Z$ with $0 \leq k \leq n$ and let $i,j \in \Z$ with $i,j \geq 0$.
$$
\text{}^{k} E_{1}^{i,k-j} := \begin{cases}
H^{i} ( D_{+}^{[k]}, \mathbb{Q} ) & \text{if} \quad j=0, \\[0.25cm]
H^{i-2j} ( D^{[k+j]}, \mathbb{Q} ) & \text{if} \quad j>0.
\end{cases}
$$
\end{defi}

Note that for $j=0$ the divisor $D_{+}$ is used, while for $j>0$ the divisor $D$ is taken. All the spaces whose cohomology is considered are compact.
These spaces give rise to the first term $E_1$ of our spectral sequence~$E = \{ E^{p,q}_n \}$:
$$
  E_{1}^{p,q} := \bigoplus_{k=0}^{n} \text{}^{k} E_1^{p,q},
$$
where $^{k} E_{1}^{p,q} = 0$ if it is not defined previously.

Note that the space $^{p} E_1^{i,k-j}$ possesses a natural pure Hodge structure of weight $i-2j$, since it is defined as the cohomology of degree $i-2j$ of a compact Kähler $V$-manifold. Performing an index shifting $\widetilde{H}^{p+j,q+j}:= H^{p,q}$, $^{p} E_1^{i,k-j}$ also has a pure Hodge structure of weight $i$, cf.~Theorem~\ref{main_th_steenbrink}.

\vspace{0.25cm}

It still remains to define the differentials. In the first term $E_1$ the differentials are of type $(0,1)$, i.e.~upward vertical arrows.


Let us resume the notation above. Let
$$
\Big( \kappa_{i_0, \ldots, i_k}^{i'_j} \Big)_{*}\, :\,
H_{*} \Big( D_{i_0,\ldots,i_j,i'_j,i_{j+1},\ldots,i_k}, \mathbb{Q} \Big) \longrightarrow H_{*} \Big( D_{i_0,\ldots,i_j,i_{j+1},\ldots,i_k}, \mathbb{Q} \Big)
$$
be the homomorphism induced by the inclusion on the homology groups. Using Poincaré duality for compact $V$-manifolds, one has the following Gysin-type maps:
$$
\xymatrix{
\ar @{} [dr] |{\#} H_{*} \Big( D_{i_0,\ldots,i_j,i'_j,i_{j+1},\ldots,i_k}, \mathbb{Q} \Big) \ar[r]^{\Big( \kappa_{i_0, \ldots, i_k}^{i'_j} \Big)_{*}} \ar[d]_{DP}^{\cong} & H_{*} \Big( D_{i_0,\ldots,i_j,i_{j+1},\ldots,i_k}, \mathbb{Q} \Big) \ar[d]^{DP}_{\cong} \\
H^{2(n-k-1)-*} \Big( D_{i_0,\ldots,i_j,i'_j,i_{j+1},\ldots,i_k}, \mathbb{Q} \Big) \ar@{-->}[r] & H^{2(n-k)-*} \Big( D_{i_0,\ldots,i_j,i_{j+1},\ldots,i_k}, \mathbb{Q} \Big) \\
}
$$

These arrows are only possible if the spaces are compact. It is always the case except for $k=0$ and $i_0=0$, where the corresponding map is defined as zero.
By abuse of notation, the morphism associated with the dashed arrow which completes the previous diagram is again denoted by $\big( \kappa_{i_0, \ldots, i_k}^{i'_j} \big)_{*}$.

\begin{defi}
The differentials on $^{k} E_1$, $\text{}^{k} \delta : \text{}^{k} E_1^{i,k-j-1} \to \text{}^k E_1^{i,k-j}$ are defined by
$$
  \text{}^{k} \delta\, |_{H^{i-2(j+1)}(D_{i_0,\ldots,i_{k+j+1}},\mathbb{Q})} := \sum_{l=0}^{k+j+1} (-1)^l \Big( \kappa_{i_0,\ldots,\widehat{i_l},\ldots,i_{k+j+1}}^{i_l} \Big)_{*}.
$$
\end{defi}

\begin{remark}
The pair $(\text{}^k E_1, \text{}^k \delta)$ is the term $E_1$ of the spectral sequence that provides the MHS of
$$
  \bigsqcup_{0 \leq i_0 < \cdots < i_k \leq s} \check{D}_{i_0,\ldots,i_k},
$$
which is the complement of a divisor with normal crossings on a projective variety.
\end{remark}

To finish with the description of the differentials, the interactions between different $\text{}^{k} E_1$ have to be taken into account. These differentials are of Mayer-Viétoris type. Denote by $\big( \kappa_{i_0,\ldots,i_{k+j}}^{i_l}
\big)^{*}$ the corresponding homomorphism on the cohomology groups.

\begin{defi}
The morphisms $\text{}^{k,k+1} \delta : \text{}^{k} E_1^{i,k-j} \to \text{}^{k+1} E_1^{i,k-j+1}$ are defined as
$$
  \text{}^{k,k+1} \delta\, |_{H^{i-2j} (D_{i_0,\ldots,i_{k+j}},\mathbb{Q})} := \sum_{\l \neq i_0,\ldots,i_{k+j}} (-1)^{e(l;\, i_0,\ldots,i_{k+j})} \Big( \kappa_{i_0,\ldots,i_{k+j}}^{i_l} \Big)^{*},
$$
where $e(l;\, i_0,\ldots,i_{k+j})$ is the number of coefficients $i_0,\ldots,i_{k+j}$ less than $l$.
\end{defi}

\begin{remark}
The pair $(\text{}^k E_1^{i,k}, \text{}^{k,k+1} \delta)$ is exactly the term $E_1$ of the spectral sequence providing the MHS of the divisor with normal crossings~$D_{+}$ which appears in~\cite{Deligne71b}. Observe that the first two columns of this spectral sequence for $k=0$ coincides with the first two columns of the term $E_1$ of $\{ E^{p,q}_{n} \}$.
\end{remark}

\begin{defi}
The direct sum of the differentials $\text{}^k \delta$ and $\text{}^{k,k+1} \delta$ is the differential $\delta$ of the term $E_1$.
\end{defi}

\begin{figure}[ht]
\centering
\includegraphics{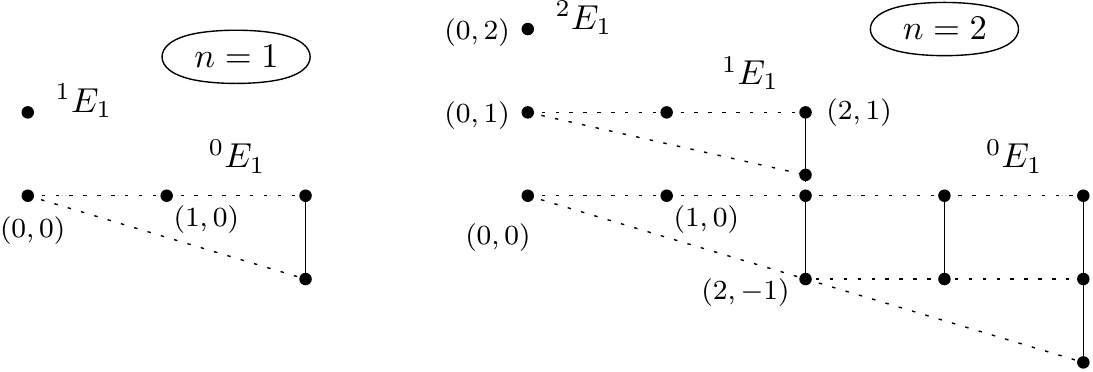}
\caption{Decomposition of $E = \{ E^{p,q}_n \}$ for $n=1,2$.}
\end{figure}

As a consequence of this spectral sequence, the standard result
about the maximal size of the Jordan blocks holds for embedded
$\Q$-resolutions too.

\begin{prop}\label{maximal_Jordan_blocks}
Let $\widetilde{K}$ be the dual complex associated with $D_{+}$
and let $\varphi: H^n(F,\C) \to H^n(F,\C)$ be the monodromy of
$\{ f = 0 \} \subset X(\bd;A)$.
Then the Jordan blocks of size
$n+1$ is determined by the characteristic polynomial of $\varphi$
acting on $H^n (\widetilde{K},\C)$.
\end{prop}

\begin{proof}
Recall that $E_1^{0,q} = H^{0}(D_{+}^{[q]},\C)$. Hence the first column of the 
spectral sequence $(E_1^{0,\bullet},\delta)$
is isomorphic to the cochain complex
$C^{\bullet} (\widetilde{K}) \otimes \C$.
Consequently $H^n (\widetilde{K},\C) \cong H^n(E_1^{0,\bullet})$; besides
the action of $\varphi$
on both complexes commutes with this isomorphism.
On the other hand, since $W$ coincides with the monodromy filtration,
the Jordan blocks of size $n+1$ are determined by the action of $\varphi$
on $\gr^W_0 ( H^n (F,\C) )$ which is by Theorem~\ref{main_th_steenbrink}
isomorphic to $E_{\infty}^{0,n} = E_2^{0,n} = H^n (E_1^{0,\bullet})$.
The latter isomorphism is again compatible with the action of $\varphi$.
This concludes the result.
\end{proof}

\begin{remark}
Analogously, one shows that $H^q(\widetilde{K},\C) \cong \gr^W_0 ( H^q(F,\C) )$
and thus
$H^0(\widetilde{K},\C) = \C$ and $H^q(\widetilde{K},\C) = 0$ for $q \neq 0,n$.
\end{remark}

This section ends with the explicit description of the spectral sequence $\{ E^{p,q}_n \} \otimes_{\mathbb{Q}} \C$ for the cases $n=1,2$. For $n=1$, let us denote with a triangle the terms belonging to $\text{}^1 E_1$ and with a circle the ones belonging to $\text{}^{0} E_1$.

\begin{figure}[ht]
\begin{small}
\centerline{
\xymatrix{
\blacktriangle \, H^{0} ( D_{+}^{[1]}, \C ) & (k=1) \\
\bullet \, H^{0} ( D_{+}^{[0]}, \C ) \ar[u]^{\text{}^{0,1} \delta} \ar@{.}[r] \ar@<-5pt>@{.}[rrd]
& \bullet \, H^1 ( D_{+}^{[0]}, \C ) \ar@{.}[r] & \bullet \, H^2 ( D_{+}^{[0]}, \C ) \\
& (k=0) & \bullet \, H^{0} ( D^{[1]}, \C ) \ar[u]_{\text{}^{0} \delta}
}}
\end{small}
\end{figure}

For surfaces, that is $n=2$, denote with a square the terms belonging to $\text{}^2 E_1$, with a triangle the ones belonging to $\text{}^1 E_1$, and finally with a circle those coming from $\text{}^{0} E_1$.

\begin{figure}[ht]
\begin{small}
\centerline{
\xymatrix@C=10pt{
\blacksquare \, H^{0}(D_{+}^{[2]}) & (k=2) \\
\blacktriangle \, H^{0}(D_{+}^{[1]}) \ar[u]^{\text{}^{1,2} \delta} \ar@{.}[r] \ar@<1pt>@{.}[drr]
& \blacktriangle \, H^{1}(D_{+}^{[1]}) \ar@{.}[r] & \blacktriangle \, H^{2}(D_{+}^{[1]}) & (k=1) \\
\bullet \, H^{0}(D_{+}^{[0]}) \ar[u]^{\text{}^{0,1} \delta} \ar@{.}[r] \ar@<-10pt>@{.}[rrrrdd]
& \bullet \, H^{1}(D_{+}^{[0]}) \ar[u]^<<<<<{\text{}^{0,1} \delta} \ar@{.}[rr]|{\hspace{1.0cm}} & \mbox{$\begin{array}{c} \blacktriangle \, H^{0} (D^{[2]}) \\ \oplus \\ \bullet \, H^{2}(D_{+}^{[0]}) \end{array}$} \ar[u]_<<<<{\text{}^{1} \delta\, \oplus\, \text{}^{0,1} \delta} & \bullet \, H^{3}(D_{+}^{[0]}) \ar@{.}[r] & \bullet \, H^{4}(D_{+}^{[0]}) \\
& (k=0) & \bullet \, H^{0}(D^{[1]}) \ar[u]^<<<<<{\text{}^{0} \delta} \ar@{.}[r] & \bullet \, H^{1}(D^{[1]}) \ar[u]^{\text{}^{0} \delta} \ar@{.}[r] & \bullet \, H^{2}(D^{[1]}) \ar[u]_{\text{}^{0} \delta} \\
&&&& \bullet \, H^{0}(D^{[2]}) \ar[u]_{\text{}^{0} \delta}
}}
\end{small}
\end{figure}

\section{Examples}\label{sec:examples}

As an application we illustrate the use of all the preceding results presented in this work with several examples including a plane curve and a YLS.
In particular, we provide infinite pairs of irreducible YLS having the same
complex monodromy with different topological type.

\subsection{Plane Curves}\label{subsec:ex_plane_curve}

Assume $\gcd(p,q) = \gcd(r,s) = 1$ and $\frac{p}{q} < \frac{r}{s}$. Let $f = (x^p + y^q) (x^r + y^s)$ and consider ${\bf C}_1 = \{ x^p + y^q = 0 \}$ and ${\bf C}_2 = \{ x^r + y^s = 0 \}$. An embedded $\Q$-resolution of $\{ f = 0 \} \subset \C^2$
is calculated in~\cite[Ex.~4.8]{AMO11b} from a $(q,p)$-blow-up at the origin of $\C^2$, followed by the $(s,qr-ps)$-blow-up at a point of type $(q;-1,p)$.
The final situation is shown in Figure~\ref{fig_blowup_curve3}.

\begin{figure}[ht]
\centering\includegraphics{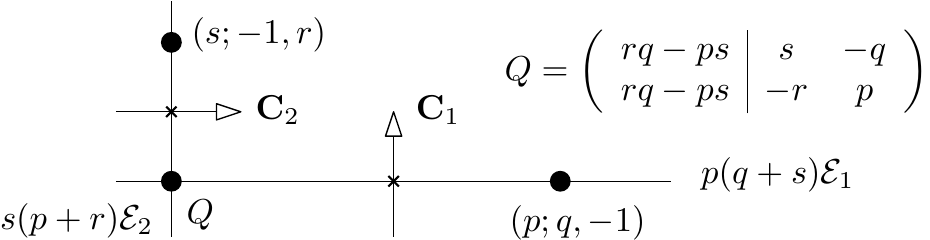}
\caption{Embedded $\Q$-resolution of $f=(x^p+y^q)(x^r+y^s)$.}
\label{fig_blowup_curve3}
\end{figure}

The self-intersection numbers are calculated using~\cite[Prop.~7.3]{AMO11b} and the intersection matrix is
$A = \frac{1}{rq-ps} \left(\begin{smallmatrix} -r/p & 1 \\ 1 & -q/s \end{smallmatrix}\right)$. By~\cite[Th.~2.8]{Martin11}, the characteristic polynomial is
$$
  \Delta (t) = \frac{\big( t-1 \big) \big( t^{p(q+s)}-1 \big) \big( t^{s(p+r)}-1 \big)}{\big( t^{q+s}-1 \big) \big( t^{p+r}-1 \big)}.
$$

The semistable reduction is studied using Theorem~\ref{SE_reduction}.
The main relevant data to compute are $m(E,Q)$
and the genera $g_1$ and $g_2$ of the new exceptional
divisors $D_1$ and $D_2$ in the semistable reduction. As explained
in~\cite{AMO11b}, the equation of the total transform at $Q$
is $x^{p(q+s)} y^{s(p+r)}$ in the quotient space of
Figure~\ref{fig_grafo-res2_SE}. By Lemma~\ref{prime_factors},
$$
m(E,Q) = \gcd \Big( p(q+s), s(p+r), A, B \Big), 
$$
where
$A = \frac{p(q+s) \cdot s + s(p+r) \cdot (-q)}{rq-ps} = - s$ and
$B = \frac{p(q+s) \cdot (-r) + s(p+r) \cdot p}{rq-ps} = - p$.
Consequently, $m(E,Q) = \gcd(p,s)$.

The restriction $\varrho\,|: D_1 \to \mathcal{E}_1$ is a branched covering
of $p(q+s)$ sheets ramifying over $3$ points, where
the number of preimages are $\gcd(p,s)$, $1$, and $q+s$. Analogous situation holds for $D_2$. Hence, by virtue of the Riemann-Hurwitz formula, the genera are
$$
g_1=\displaystyle\frac{(p-1)(q+s)-\gcd(p,s)+1}{2}, \quad
g_2=\displaystyle\frac{(s-1)(p+r)-\gcd(p,s)+1}{2}.
$$

The dual graph of the new normal crossing divisor after the semistable reduction
process is shown in Figure~\ref{fig_grafo-res2_SE}.

\begin{figure}[h t]
\centering
\includegraphics{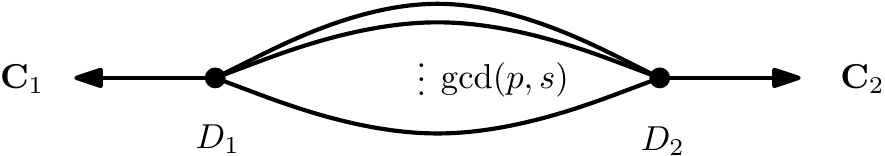}
\caption{Dual graph of the semistable reduction.}
\label{fig_grafo-res2_SE}
\end{figure}

The MHS on the cohomology of the Milnor fiber $H^1(F,\C)$ is obtained from Steenbrink's spectral sequence:
$$
  H^1 (F, \C) = \underbrace{H^{0,0}}_{\gr^W_0 H^1 (F,\C)} \oplus \quad \underbrace{H^{0,1} \oplus H^{1,0}}_{\gr^W_1 H^1 (F,\C)} \quad \oplus \underbrace{H^{1,1}}_{\gr^W_2 H^1 (F,\C)},
$$
where $H^{0,0} = \C^{\gcd(p,s)-1}$, $H^{0,1} = \C^{g_1+g_2} = \overline{H^{1,0}}$,
and $H^{1,1} = \C^{\gcd(p,s)}$.

The action of the monodromy on $\gr^W_0 H^1 (F, \C)$ is given by the polynomial $\frac{t^{\gcd(p,s)}-1}{t-1}$. Note that this provides the eigenvalues of the monodromy with Jordan blocks of size $2$. This has to do with the fact 
that the dual graph possesses $\gcd(p,s)-1$ cycles,
see Proposition~\ref{maximal_Jordan_blocks},
cf.~\cite[Ch.~V.4]{Martin11PhD} for a more detailed exposition.
Also, note that this example has already been treated in~\cite{Grima74} for the 
cases $(p,q,r,s) = (21,44,14,11), (33,28,22,7)$; they both have the same monodromy
but different topological type as one can easily check.

\begin{remark}
The previous example is generalized to several branches with no significant
changes. Let $f = ( x^{p_1} + y^{q_1} ) \cdots ( x^{p_k} + y^{q_k} )$,
$k \geq 1$, $\frac{p_1}{q_1} < \cdots < \frac{p_k}{q_k}$, and $p_i, q_i \geq 1$
no necessarily coprime, $d_i = \gcd(p_i,q_i)$.

\begin{figure}[ht]
\centering
\includegraphics{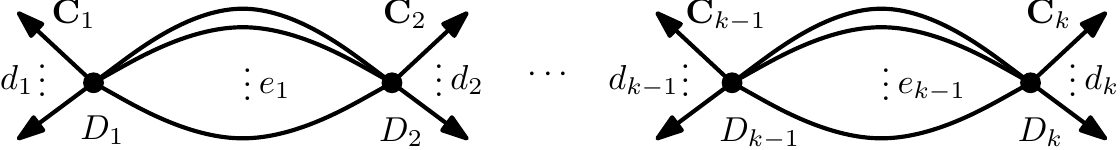}
\caption{Dual graph of the semistable reduction of $f$.}
\label{fig_grafo-res2_SE_general}
\end{figure}

Denote $e_i = \gcd( p_1 + \cdots + p_i, q_{i+1} + \cdots + q_k )$, $i=1,\ldots,k-1$.
Then the Jordan blocks of size $2$ is given by the polynomial $\prod_{i=1}^{k-1}
(t^{e_i}-1) / (t-1)$. The dual graph of the semistable reduction is shown in
Figure~\ref{fig_grafo-res2_SE_general}.
\end{remark}

\subsection{Yomdin-L\^{e} Surface Singularities}

Let $(V,0)$ be the singularity defined by $f = f_m(x,y,z) + z^{m+k}$. Assume that ${\bf C} = \{f_m = 0\} \subset \P^2$ has only one singular point $P = [0:0:1]$, which is locally isomorphic to the cusp $x^q+y^p$, $\gcd(p,q)=1$. Consider the weight vector $\w = (\frac{k p}{k_1 k_2}, \frac{k q}{k_1 k_2}, \frac{p q}{k_1 k_2})$, where $k_1= \gcd(k,p)$ and $k_2 = \gcd(k,q)$.

An embedded $\Q$-resolution of $\{ f = 0 \} \subset \C^3$
is calculated in~\cite{Martin12}. It is required to perform first the standard
blow-up at the origin of~$\C^3$ and then the $\w$-blow-up at the point
$P=[0:0:1] \in \mathbf{C} \cap E_1$, where the total transform is not a normal crossing divisor. Denote by $E_1 := \widehat{E}_1$ and by $E_2$ the second exceptional divisor. The final situation is shown in Figure~\ref{fig_acampo6}.

\begin{figure}[h]
\centering
\includegraphics{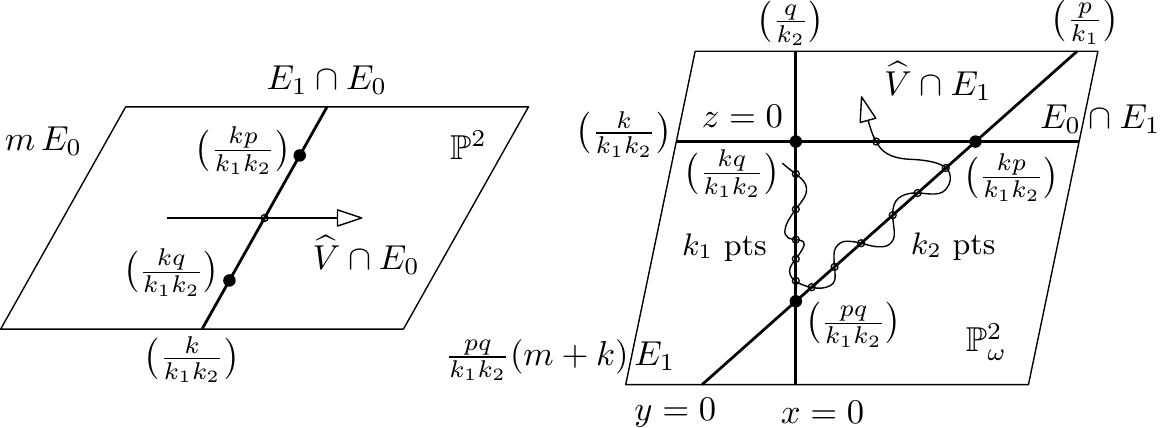}
\caption{Intersection of $E_0$ and $E_1$ with the rest of components.}
\label{fig_acampo6}
\end{figure}

Applying the generalized A'Campo's formula \cite[Th.~2.8]{Martin11}, the characteristic polynomial of $(V,0)$ is
$$
\Delta_{(V,0)}(t)
= \frac{\big(t^m-1\big)^{\chi(\P^2\setminus{\bf C})}}{t-1} \cdot \frac{\big(t^{m+k}-1 \big) \big( t^{\frac{p q}{k_1 k_2}(m+k)}-1 \big)^{k_1 k_2}}{\big( t^{\frac{p}{k_1}(m+k)}-1 \big)^{k_1} \big( t^{\frac{q}{k_2}(m+k)}-1 \big)^{k_2}}.
$$

The semistable reduction is studied using Theorem~\ref{SE_reduction}.
We only discuss the action of the monodromy on $\gr^W_1 \! H$ and
$\gr^W_4 \! H$, being $H:=H^2(F,\C)$, which encode the $2$ and $3$-Jordan blocks.
The preimages of $\widehat{V}$, $E_0$, $E_1$ under $\varrho$ are denoted by
$\widehat{V}$, $D_0$, $D_1$; they are all irreducible varieties.

The fifth column of the generalized Steenbrink's spectral sequence gives rise
the following exact sequence of vector spaces
$$
0 \longrightarrow \gr^W_4\!H \longrightarrow H^0(D^{[2]}) \longrightarrow
H^2 (D^{[1]}) \longrightarrow H^4(D^{[0]}_{+}) \longrightarrow 0.
$$

Note that $D_0$, $D_1$, $\widehat{V} \cap D_0$, $\widehat{V} \cap D_1$,
$D_0 \cap D_1$, and $\widehat{V} \cap D_0 \cap D_1$ are all irreducible
varieties because they intersect $\widehat{V}$. Thus $h^4(D^{[0]}_{+})
= h^0(D^{[0]}_{+}) = 2$, $h^2(D^{[1]}) = h^0(D^{[1]}) =3$, and
$h^0(D^{[2]}) = 1$. Therefore $\gr^W_4\!H$ is trivial and then there are neither
$2$-Jordan blocks for $\lambda=1$ nor $3$-Jordan blocks ($\lambda \neq 1$).

From the second column of the spectral sequence,
$$
  0 \longrightarrow H^1(D^{[0]}_{+}) \longrightarrow H^1(D^{[1]}_{+})
  \longrightarrow \gr^W_1 \! H \longrightarrow 0.
$$

The restriction $\varrho\,|:D_1 \to E_1 \cong \P^2_{\w}$ is a branched covering of
$(m+k)\frac{pq}{k_1k_2}$ sheets ramifying over the axes and the curve
$\widehat{V} \cap E_1 = \{ x^q + y^p + z^k = 0 \}$. The composition of the previous
map with $E_1 \to \P^2$, $[x:y:z]_{\w} \mapsto [x^q:y^p:z^k]$, is an
abelian covering ramifying over $4$ lines in general position. This implies
$H^1(D_1) = 0$. On the other hand, note that the cohomology $H^1(D_0)$ is
determined by the pair $(\P^2,\mathbf{C})$ and hence so is
$H^1(D^{[0]}_{+}) = H^1(D_0)$.

Finally, the first cohomology of the Riemann surface $D^{[1]}_{+} = D_0 \cap D_1$
is studied. Using Lemma~\ref{prime_factors}, one checks that,
$m(\widehat{V} \cap E_0 \cap E_1) = 1$, 
$m(\text{``generic point of $E_0\cap E_1$''}) = \gcd(m,pq)$, and
$$
m\left(E,\left(\frac{kq}{k_1k_2}\right)\right) = \gcd(m,p), \quad
m\left(E,\left(\frac{kp}{k_1k_2}\right)\right) = \gcd(m,q).
$$
This means that $\varrho\,|:D_0 \cap D_1 \to E_0 \cap E_1$ is a branched
covering of $\gcd(m,pq)$ sheets ramifying over $3$ points where the number
of preimages are $\gcd(m,p)$, $1$, and $\gcd(m,q)$. It follows that
$$
\Delta_{\gr^W_1 H} (t) =
\frac{1}{\Delta_{H^1(D_0)}(t)}
\cdot \frac{\left( t-1 \right) \left( t^{\gcd(m,pq)}-1 \right)}
{\left( t^{\gcd(m,p)}-1 \right) \left( t^{\gcd(m,q)}-1 \right)}. 
$$

\begin{remark}
The singularity of the tangent cone in the previous example is so simple
that the Jordan blocks of size $2$ and $3$ (for $\lambda \neq 1$ or $\lambda =1$)
do not depend on $k$. However, this is not true in general, see below.
\end{remark}

Let $V$ be the YLS defined by $f = f_m(x,y,z) + z^{m+k}$ where ${\bf C} = \{f_m = 0\} \subset \P^2$ has only one singular point $P = [0:0:1]$, which is locally isomorphic to $(x^p+y^q)(x^r+y^s)$ with $\gcd(p,q)=\gcd(r,s)=1$ and $\frac{p}{q}
< \frac{r}{s}$, cf.~\S\ref{subsec:ex_plane_curve}.
Using the techniques presented in this paper and the $\Q$-resolution 
calculated in \cite{Martin12}, we were able to compute the following:
$$
\Delta_{\gr^W_1 \! H}(t) = \frac{\displaystyle \frac{\big( t^{(m,p(q+s))}-1 \big) \big(t^{(m,s(p+r))}-1\big)}{\big(t^{(m,q+s)}-1\big)\big(t^{(m,p+r)}-1\big)}
\cdot \frac{\big(t^{(m+k)\frac{(p,s)}{(k,p,s)}}-1\big)^{(k,p,s)}}{(t^{m+k}-1)
(t-1)^{(k,p,s)-1}}}{\displaystyle\Delta_{H^1(D_0)}(t) \cdot \Bigg( \frac{t^{(m,p,s)}-1}{t-1} \Bigg)^{3}},
$$
$$
\Delta_{\gr^W_4 \! H}(t) = \frac{t^{(m,p,s)}-1}{t-1} \cdot (t-1)^{(k,p,s)},
$$
where for simplicity $(a,b)$ denotes $\gcd(a,b)$. Note that the cohomology $H^1(D_0)$
is determined by the pair $(\P^2,\mathbf{C})$, the first factor of the numerator
has to do with the characteristic polynomial of the tangent cone at $[0:0:1]$, and
the second factor of both the numerator and denominator is related to the $2$-Jordan
blocks of the tangent cone, see Example in \S\ref{subsec:ex_plane_curve}.


In particular, considering $(p,q,r,s) = (21,44,14,11)$, $(33,28,22,7)$ and
$m$ generic so that $\Delta_{H^1(D_0)}(t)=1$,
one obtains infinite pairs of irreducible YLS having the same complex
monodromy and different topological type. Examples of this kind have already
been found, for instance, in~\cite{Artal91} studying the associated Seifert form,
which determines the integral monodromy.
We ignore whether our preceding example have the same integral monodromy or
the same Seifert form.

\def\cprime{$'$}


\begin{thebibliography}{10}
\bibitem{ACampo75}
N.~A'Campo.
\newblock La fonction z\^eta d'une monodromie.
\newblock {\em Comment. Math. Helv.}, 50:233--248, 1975.

\bibitem{AVG88}
V.~I. Arnol'd, S.~M. Guse{\u\i}n{-}Zade, and A.~N. Varchenko.
\newblock {\em Singularities of differentiable maps. {V}ol. {II}}, volume~83 of
  {\em Monographs in Mathematics}.
\newblock Birkh\"auser Boston Inc., Boston, MA, 1988.
\newblock Monodromy and asymptotics of integrals, Translated from the Russian
  by Hugh Porteous, Translation revised by the authors and James Montaldi.

\bibitem{Artal94}
E.~Artal{ }Bartolo.
\newblock Forme de {J}ordan de la monodromie des singularit\'es superisol\'ees
  de surfaces.
\newblock {\em Mem. Amer. Math. Soc.}, 109(525):x+84, 1994.

\bibitem{AMO11b}
E.~{Artal Bartolo}, J.~{Martín-Morales}, and J.~{Ortigas-Galindo}.
\newblock Intersection theory on abelian-quotient {$V$}-surfaces and
  {$\mathbf{Q}$}-resolutions.
\newblock {\em ArXiv e-prints}, May 2011.

\bibitem{Artal91}
Enrique Artal-Bartolo.
\newblock Forme de {S}eifert des singularit\'es de surface.
\newblock {\em C. R. Acad. Sci. Paris S\'er. I Math.}, 313(10):689--692, 1991.

\bibitem{Baily56}
W.~L. Baily.
\newblock The decomposition theorem for {$V$}-manifolds.
\newblock {\em Amer. J. Math.}, 78:862--888, 1956.

\bibitem{CMO12}
J.~I. {Cogolludo-Agust\'in}, J.~{Mart\'in-Morales}, and J.~{Ortigas-Galindo}.
\newblock Local invariants on quotient singularities and a genus formula for
  weighted plane curves.
\newblock {\em Int. Math. Res. Notices}, 2013.
\newblock DOI: 10.1093/imrn/rnt052.

\bibitem{Deligne71a}
P.~Deligne.
\newblock Th\'eorie de {H}odge. {I}.
\newblock In {\em Actes du {C}ongr\`es {I}nternational des {M}ath\'ematiciens
  ({N}ice, 1970), {T}ome 1}, pages 425--430. Gauthier-Villars, Paris, 1971.

\bibitem{Deligne71b}
P.~Deligne.
\newblock Th\'eorie de {H}odge. {II}.
\newblock {\em Inst. Hautes \'Etudes Sci. Publ. Math.}, (40):5--57, 1971.

\bibitem{Deligne74}
P.~Deligne.
\newblock Th\'eorie de {H}odge. {III}.
\newblock {\em Inst. Hautes \'Etudes Sci. Publ. Math.}, (44):5--77, 1974.

\bibitem{GS75}
P.~Griffiths and W.~Schmid.
\newblock Recent developments in {H}odge theory: a discussion of techniques and
  results.
\newblock In {\em Discrete subgroups of {L}ie groups and applicatons to moduli
  ({I}nternat. {C}olloq., {B}ombay, 1973)}, pages 31--127. Oxford Univ. Press,
  Bombay, 1975.

\bibitem{Grima74}
M.-C. Grima.
\newblock La monodromie rationnelle ne d\'etermine pas la topologie d'une
  hypersurface complexe.
\newblock In {\em Fonctions de plusieurs variables complexes ({S}\'em. {F}ran\c
  cois {N}orguet, 1970--1973; \`a la m\'emoire d'{A}ndr\'e {M}artineau)}, pages
  580--602. Lecture Notes in Math., Vol. 409. Springer, Berlin, 1974.

\bibitem{GLM97}
S.~M. Gusein{-}Zade, I.~Luengo, and A.~Melle{-}Hern{\'a}ndez.
\newblock Partial resolutions and the zeta-function of a singularity.
\newblock {\em Comment. Math. Helv.}, 72(2):244--256, 1997.

\bibitem{Mumford73}
G.~Kempf, F.~F. Knudsen, D.~Mumford, and B.~Saint{-}Donat.
\newblock {\em Toroidal embeddings. {I}}.
\newblock Lecture Notes in Mathematics, Vol. 339. Springer-Verlag, Berlin,
  1973.

\bibitem{Martin11PhD}
J.~{Mart{\'{\i}}n-Morales}.
\newblock {\em Embedded $\mathbf{Q}$-{R}esolutions and {Y}omdin-{L}\^{e}
  {S}urface {S}ingularities}.
\newblock {PhD} dissertation, {IUMA}-University of {Z}aragoza, December 2011.
\newblock URL: \texttt{http://cud.unizar.es/martin}.

\bibitem{Martin12}
J.~{Martín-Morales}.
\newblock Embedded {$\mathbf{Q}$}-resolutions for {Y}omdin-{L}ê surface
  singularities.
\newblock {\em \emph{Accepted in} Isr. J. Math.}, 2013.
\newblock Preprint at arXiv:1206.0454.

\bibitem{Martin11}
J.~{Martín-Morales}.
\newblock Monodromy zeta function formula for embedded
  {$\mathbf{Q}$}-resolutions.
\newblock {\em Rev. Mat. Iberoam.}, 29(3):939--967, 2013.

\bibitem{Satake56}
I.~Satake.
\newblock On a generalization of the notion of manifold.
\newblock {\em Proc. Nat. Acad. Sci. U.S.A.}, 42:359--363, 1956.

\bibitem{Schmid73}
W.~Schmid.
\newblock Variation of {H}odge structure: the singularities of the period
  mapping.
\newblock {\em Invent. Math.}, 22:211--319, 1973.

\bibitem{Serre66}
J.-P. Serre.
\newblock {\em Alg\`ebres de {L}ie semi-simples complexes}.
\newblock W. A. Benjamin, inc., New York-Amsterdam, 1966.

\bibitem{Steenbrink77}
J.~H.~M. Steenbrink.
\newblock Mixed {H}odge structure on the vanishing cohomology.
\newblock In {\em Real and complex singularities ({P}roc. {N}inth {N}ordic
  {S}ummer {S}chool/{NAVF} {S}ympos. {M}ath., {O}slo, 1976)}, pages 525--563.
  Sijthoff and Noordhoff, Alphen aan den Rijn, 1977.

\bibitem{Varchenko80}
A.~N. Var{\v{c}}enko.
\newblock Asymptotic behavior of holomorphic forms determines a mixed {H}odge
  structure.
\newblock {\em Dokl. Akad. Nauk SSSR}, 255(5):1035--1038, 1980.

\bibitem{Varchenko81}
A.~N. Var{\v{c}}enko.
\newblock Asymptotic {H}odge structure on vanishing cohomology.
\newblock {\em Izv. Akad. Nauk SSSR Ser. Mat.}, 45(3):540--591, 688, 1981.

\bibitem{Yomdin74}
Y.~Yomdin.
\newblock Complex surfaces with a one-dimensional set of singularities.
\newblock {\em Sibirsk. Mat. \v Z.}, 15:1061--1082, 1181, 1974.
\end{thebibliography}
\end{document}